\documentclass[draftclsnofoot,onecolumn]{IEEEtran}
\usepackage{def}
\begin{document}

\title{Smoothness of the Augmented Lagrangian Dual in Convex Optimization}

\author{Jingwang Li and Vincent Lau, \IEEEmembership{Fellow, IEEE}
  \thanks{The authors are with the Department of Electronic and Computer Engineering, The Hong Kong University of Science and Technology, Hong Kong 99077, China (e-mail: jingwangli@outlook.com; eeknlau@ust.hk).}
  \thanks{\textit{Corresponding author: Vincent Lau.}}
}

\maketitle

\begin{abstract}
  This paper focuses on the general linearly constrained optimization problem: $\min_{x \in \mathbb{R}^d} f(x) \ \text{s.t.} \ Ax = b$, where $f: \mathbb{R}^d \rightarrow \mathbb{R} \cup \{+\infty\}$ is a closed proper convex function, $A \in \mathbb{R}^{p \times d}$, and $b \in \mathbb{R}^p$. We define the standard dual function $\phi(\lambda) = \inf_x \{f(x) + \langle \lambda, A x - b \rangle\}$, the augmented Lagrangian $\mathcal{L}_{\rho}(x, \lambda) = f(x) + \langle \lambda, Ax - b \rangle + \frac{\rho}{2}\|Ax - b\|^2$ ($\rho > 0$), and the augmented Lagrangian dual function $\phi_{\rho}(\lambda) = \inf_x \mathcal{L}_{\rho}(x, \lambda)$. Under the fundamental condition that $\text{dom} \ \phi \neq \emptyset$, we establish that: (1) $\phi_{\rho}$ is $\frac{1}{\rho}$-smooth everywhere; and (2) the solution to $\min_{x \in \mathbb{R}^d} \mathcal{L}_{\rho}(x, \lambda)$ exists for any $\lambda \in \mathbb{R}^p$. These theoretical findings substantially weaken the stringent assumptions typically imposed in the literature to ensure such properties.
\end{abstract}


\IEEEpeerreviewmaketitle

\section{Introduction}
In this paper, we consider the general linearly constrained optimization problem:
\begin{equation} \label{main_pro} \tag{P1}
  \begin{aligned}
  \min_{x \in \R^d} \ &f(x) \\
  \st \ &A x = b,
\end{aligned}
\end{equation}
where $f: \R^d \rightarrow \exs$ is a convex but possibly nonsmooth function, $A \in \R^{p \times d}$, and $b \in \R^p$. Without loss of generality, we assume that the solution to \cref{main_pro} exists. The augmented Lagrangian method (ALM), originally proposed by Hestenes \cite{hestenes1969multiplier} and Powell \cite{powell1969method}, remains a cornerstone algorithm for solving \cref{main_pro}. Its iterative scheme is given by
\eqa{
  x\+ &= \arg\min_{x \in \R^d} \cL_{\rho}(x, \lambda^k), \label{ALM_p} \\
  \lambda\+ &= \lambda^k + \rho (A x\+ - b), \label{ALM_d}
}
where $\lambda \in \R^p$ is the dual variable, $\rho > 0$ is the augmented parameter, and the augmented Lagrangian $\cL_{\rho}$ is defined as
\eqe{
  \cL_{\rho}(x, \lambda) &= f(x) + \dotprod{\lambda, A x - b} + \frac{\rho}{2}\norm{A x - b}^2.
}

One convenient approach for analyzing the iteration complexity of (inexact) ALM\footnote{For inexact ALM, the subproblem in \cref{ALM_p} is allowed to be solved inexactly.} \cite{nedelcu2014computational,lan2016iteration} lies in interpreting it as (inexact) gradient descent (GD) \cite{devolder2014first,schmidt2011convergence} applied to the augmented Lagrangian dual problem of \cref{main_pro}:
\eqe{ \label{dual_pro}
  \max_{\lambda \in \R^p} \ \phi_{\rho}(\lambda),
}
where the augmented Lagrangian dual function $\phi_{\rho}$ is defined as
\eqe{ \label{1522}
  \phi_{\rho}(\lambda) = \inf_x \cL_{\rho}(x, \lambda) = -f_{\rho}^*\pa{-A\T\lambda} -\lambda\T b,
}
with $f_{\rho}(x) = f(x) + \frac{\rho}{2}\norm{A x - b}^2$. Then, the convergence and outer iteration complexity of ALM can be easily established by applying the convergence results of GD. Since the convergence of GD depends critically on the objective function's smoothness and convexity, establishing the smoothness of $\phi_{\rho}$ (which is always concave) becomes pivotal.

\textbf{Existing Results and Limitations.} For the special case where $f = g + \iota_{\sX}$ with $\sX \subseteq \R^d$ being a compact convex set and $g: \sX \rightarrow \R$ continuous convex, the Weierstrass' theorem \cite[Proposition A.8]{bertsekas2016nonlinear} guarantees the existence of solutions to $\min_{x \in \R^d} \cL_{\rho}(x, \lambda)$ for any $\lambda \in \R^p$, thus implying $\dom{\phi_{\rho}} = \R^{p}$. The $\frac{1}{\rho}$-smoothness of $\phi_{\rho}$ then follows directly from \cite[Theorem 1]{nesterov2005smooth} (see \cite[Proposition 1]{lan2016iteration}). An alternative approach for establishing the smoothness of $\phi_{\rho}$ involves demonstrating the invariance of $Ax$ over the solution set of $\min_{x \in \R^d} \cL_{\rho}(x, \lambda)$ and then applying the Danskin's Theorem \cite[Proposition B.22(b)]{bertsekas2016nonlinear} (refer to \cite[Lemmas 2.1 and 2.2]{hong2017linear}). However, this analysis fails when $\sX$ is not bounded. In such cases, additional assumptions are necessary to ensure the existence of solutions to $\min_{x \in \R^d} \cL_{\rho}(x, \lambda)$; for instance, $\cL_{\rho}(\cdot, \lambda)$ is coercive\footnote{We say a function $f$ is coercive if $\lim_{\norm{x} \rightarrow \infty}f(x) = \infty$.} (which holds if $f$ is coercive). Even with these assumptions, we cannot straightforwardly apply \cite[Theorem 1]{nesterov2005smooth} or the Danskin's Theorem to establish the smoothness of $\phi_{\rho}$.

The situation becomes even more complex when $f$ does not conform to the specific form mentioned above, and we only assume that $f$ is closed proper\footnote{We say a convex function $f$ is proper if $f(x)>-\infty$ for all $x$ and $\dom{f} \neq \emptyset$. Correspondingly, a concave function $g$ is termed proper if $-g$ is proper.} convex. To ensure the existence of solutions to $\min_{x \in \R^d} \cL_{\rho}(x, \lambda)$ and the smoothness of $\phi_{\rho}$, we may need to introduce additional assumptions. For example, assuming that $f$ is strongly convex or that $A$ has full column rank. In this scenario, $\cL_{\rho}(\cdot, \lambda)$ is strongly convex, leading to a singleton solution set for $\min_{x \in \R^d} \cL_{\rho}(x, \lambda)$, and the smoothness of $\phi_{\rho}$ can be established using results on the Fenchel conjugate from \cite[Lecture 5]{vandenberghe2022om}. However, these assumptions are too restrictive, limiting the applicability of the results. Thus, we are interested in whether it is possible to establish the aforementioned results for general closed proper convex functions $f$.

\textbf{Our Contributions.} Define the standard dual function of \cref{main_pro} as
\eqe{
  \phi(\lambda) = \inf_x\set{f(x) + \dotprod{\lambda, A x - b}}.
}
This paper resolves this gap by proving that for \textit{any} closed proper convex $f$, under the fundamental condition $\dom{\phi} \neq \emptyset$:
\begin{enumerate}
  \item The augmented Lagrangian dual function $\phi_{\rho}$ is $\frac{1}{\rho}$-smooth everywhere.
  \item The solution to $\min_{x \in \R^d} \cL_{\rho}(x, \lambda)$ exists for any $\lambda \in \R^p$.
\end{enumerate}
These results reveal an interesting fact: the augmented Lagrangian dual function is generally smooth in convex optimization, significantly broadening the theoretical foundation of ALM. Consequently, convergence rates and iteration complexity bounds for both standard and accelerated (inexact) ALM variants \cite{nesterov2018lectures,he2010acceleration,kang2013accelerated} can now be rigorously established for any problem reducible to \cref{main_pro}.

\section{Preliminaries}
\subsection{Notations}
We use the standard inner product $\dotprod{\cdot, \cdot}$ and the standard Euclidean norm $\norm{\cdot}$ for vectors, and the standard spectral norm $\norm{\cdot}$ for matrices. For a matrix $B \in \R^{m \times n}$, let $\Range{B}$ denote its column space. For a function $f: \R^n \rightarrow \exs$, $\dom{f} = \set{x \in \R^n | f(x) < +\infty}$ denotes its (effective) domain (for a concave function $g: \R^n \rightarrow \R \cup \{-\infty\}$, $\dom{g} \triangleq \dom{(-g)}$), $\partial f(x)$ denotes its subdifferential at $x$, $\prox_{\alpha f}(x) = \arg\min_{y} f(y) + \frac{1}{2\alpha}\norm{y-x}^2$ denotes the proximal operator of $f$ with $\alpha > 0$, $f^*(y) = \sup_{x \in \R^n}y\T x-f(x)$ and $f^{**}(x) = \sup_{y \in \R^n}x\T y-f^*(y)$ denote its Fenchel conjugate and Fenchel biconjugate respectively, and $M_{\gamma f}(x) = \inf_{y}\set{f(y) + \frac{1}{2\gamma}\norm{x-y}^2}$ denotes its Moreau envelope (or Moreau-Yosida regularization) with $\gamma>0$. For a set $\sX \subseteq \R^n$, $\inte{\sX}$ denotes its interior, $\ri{\sX}$ denotes its relative interior, and $\iota_{\sX}(x) = \lt\{
  \begin{aligned}
    0,      \   & x \in \sX,        \\
    +\infty, \  & \text{otherwise}.
  \end{aligned}\right.$ denotes its indicator function.

\subsection{Convex Analysis}

The following lemma demonstrates the smoothness of the Moreau envelope of any closed proper convex function.
\begin{lemma} \label{moreau} \cite[Theorem 3 of Lecture 26]{davis2016mp}
  Assume that $f: \R^n \rightarrow \exs$ is closed proper convex. Then: (1) $\dom{M_{\gamma f}} = \R^n$; (2) $M_{\gamma f}$ is continuously differentiable on $\R^n$ and $\nabla M_{\gamma f}(x) = \frac{1}{\gamma}\pa{x-\prox_{\gamma f}(x)}$; (3) $M_{\gamma f}$ is convex and $\frac{1}{\gamma}$-smooth.
\end{lemma}

\begin{lemma} \label{conjugate_subdiff} \cite[Corollary E.1.4.4]{hiriart2004fundamentals}
  Assume that $f: \R^n \rightarrow \exs$ is closed proper convex, we have
  \eqe{
    y \in \partial f(x) \iff x \in \partial f^*(y), \ \forall x, y \in \R^n.
  }
\end{lemma}

The following two lemmas pertain to subdifferential calculus. Specifically, \cref{subdiff} can be easily derived from \cite[Theorems 23.8]{rockafellar1970convex} by utilizing two key facts: (1) $\ri{\R^n} = \R^n$, and (2) $\ri{\dom{f_1}} \neq \emptyset$ if $\dom{f_1} \neq \emptyset$ \cite[Theorem 6.2]{rockafellar1970convex}.
\begin{lemma} \label{subdiff}
  Assume that $f_1: \R^n \rightarrow \exs$ is proper convex and $f_2: \R^n \rightarrow \R$ is convex, then $\partial (f_1 + f_2) = \partial f_1 + \partial f_2$.
\end{lemma}

\begin{lemma} \label{0142} \cite[Proposition 8 of Lecture 24]{davis2016mp}
  Assume that $g: \R^n \rightarrow \exs$ is closed proper convex, $A \in \R^{n \times m}$, and $b \in \R^n$. Let $f(x) = g(Ax-b)$, we have
  \eqe{
    \partial f(x) = A\T\partial g(Ax-b), \ \forall x \in \inte{\dom{f}}.
  }
\end{lemma}

\section{Main Results}
\begin{assumption} \label{f}
  $f$ is closed proper convex.
\end{assumption}

Our main results are presented in the following theorem.
\begin{theorem} \label{augmented_dual}
  Assume that \cref{f} holds, and $\dom{\phi} \neq \emptyset$. Then: (1) $\dom{\phi_{\rho}} = \R^{p}$; (2) $\phi_{\rho}$ is concave and $\frac{1}{\rho}$-smooth; (3) For any $\lambda \in \R^p$, there exists at least one solution $x^+$ of $\min_{x \in \R^d} \cL_{\rho}(x, \lambda)$. Moreover, $\nabla \phi_{\rho}(\lambda) = A x^+ - b$ holds for any solution $x^+$.
\end{theorem}

The following lemma presents some sufficient and necessary/unnecessary conditions for $\dom{\phi} \neq \emptyset$.
\begin{lemma} \label{dom_phi}
  $\dom{\phi} \neq \emptyset$ \textbf{if and only if} any of the following holds:
  \begin{enumerate}[~~(1)]
    \item there exists $\lambda \in \R^p$ such that $\phi(\lambda) > -\infty$;
    \item $\dom{f^*} \cap \Range{A\T} \neq \emptyset$.
  \end{enumerate}
  Assume that \cref{f} holds, $\dom{\phi} \neq \emptyset$ \textbf{if} any of the following holds\footnote{The examples provided below are illustrative rather than exhaustive; additional sufficient conditions can be readily derived by leveraging the two necessary and sufficient conditions.}:
  \begin{enumerate}[~~(a)]
    \item There exists $\lambda \in \R^p$ such that the solution to $\min_{x \in \R^d} f(x) + \dotprod{\lambda, Ax - b}$ exists, i.e., there exists $(x, \lambda) \in \R^d \times \R^p$ such that $\0 \in \partial f(x) + A\T\lambda$;
    \item The solution to the (standard) dual problem of \cref{main_pro}, i.e., $\max_{\lambda \in \R^p} \ \phi(\lambda)$, exists;
    \item Slater's condition holds for \cref{main_pro}, i.e., there exists $x \in \ri{dom{f}}$ such that $Ax = b$;
    \item There exists $(x, \lambda) \in \R^d \times \R^p$ that satisfies the KKT conditions of \cref{main_pro}, i.e.,
    \eqe{
      &\0 \in \partial f(x) + A\T\lambda, \\
      &Ax = b;
    }
    \item $\inf_x f(x) > -\infty$;
    \item $f$ is coercive.
  \end{enumerate}
\end{lemma}

\begin{proof}
  We first discuss the sufficient and necessary conditions. (1) is obvious, while (2) follows directly from
  \eqe{
    \phi(\lambda) = -f^*\pa{-A\T\lambda} -\lambda\T b.
  }

  For the sufficient conditions, (a) and (b) are both sufficient for (1).
  According to the existence of a solution to \cref{main_pro} and \cite[Corollary 31.2.1]{rockafellar1970convex}, (c) is sufficient for (b).
  (d) is sufficient for both (a) and (b).
  (e) is sufficient for (2) due to two facts: (i) $f^*(\0) = \sup_x \0\T x - f(x) = -\inf_x f(x) < +\infty$; and (ii) $\0 \in \Range{A\T}$.
  (f) is sufficient for (e).
\end{proof}

\begin{remark} \label{0036}
  Theorem \ref{augmented_dual} reveals that the augmented Lagrangian dual function of \cref{main_pro} is smooth everywhere only $f$ is closed proper convex and $\dom{\phi} \neq \emptyset$. This result is significantly more general than existing results, which require stronger assumptions, such as $f = g + \iota_{\sX}$, where $\sX \subseteq \R^d$ is a compact convex set and $g: \sX \rightarrow \R$ is a continuous convex function \cite{lan2016iteration,hong2017linear}.
\end{remark}

We first establish the equivalence between $M_{\rho (-\phi)}$ and $-\phi_{\rho}$, where
\eqe{
  M_{\rho (-\phi)}(\lambda) = \inf_{w}\pa{-\phi(w) + \frac{1}{2\rho}\norm{w-\lambda}^2}
}
is the Moreau envelope of $-\phi$.
\begin{lemma} \label{1951}
  Assume that \cref{f} holds, we have
  \eqe{
    M_{\rho (-\phi)}(\lambda) = -\phi_{\rho}(\lambda), \ \forall \lambda \in \R^p,
  }
  i.e., $M_{\rho (-\phi)} = -\phi_{\rho}$.
\end{lemma}

\begin{proof}
  Define
  \eqe{
    \Gamma_{\lambda}(w, x) = \frac{1}{2\rho}\norm{w-\lambda}^2 - \dotprod{w, A x - b} - f(x),
  }
  we have
  \eqe{ \label{00160}
    M_{\rho (-\phi)}(\lambda) = \inf_{w}\sup_{x} \Gamma_{\lambda}(w, x).
  }
  We first prove that
  \eqe{ \label{21450}
    \inf_{w}\sup_{x} \Gamma_{\lambda}(w, x) = \sup_{x}\inf_{w} \Gamma_{\lambda}(w, x), \ \forall \lambda \in \R^p.
  }
  Define $q(w) = \frac{1}{2\rho}\norm{w-\lambda}^2$, we can easily obtain its Fenchel conjugate: $q^*(y) = \frac{\rho}{2}\norm{y}^2 + y\T\lambda$. Clearly $\dom{q^*} = \R^p$ and $q^{**} = q$. Let $\ell(y) = y - b$, it is obvious that $\dom{\pa{q^* \circ \ell}} = \R^p$ and $q^* \circ \ell$ is closed proper convex.
  It follows that
  \eqe{ \label{23480}
    \sup_{x}\inf_{w} \Gamma_{\lambda}(w, x) &= \sup_{x}\pa{\inf_{w}\pa{\frac{1}{2\rho}\norm{w-\lambda}^2 - \dotprod{w, A x - b}} - f(x)} \\
    &= \sup_{x}\pa{-q^*\pa{A x - b} - f(x)} \\
    &= -\inf_{x}\pa{f(x) + \pa{q^* \circ \ell}(A x)}.
  }
  According to \cite[Proposition E.1.3.1]{hiriart2004fundamentals}, we have
  \eqe{
    \pa{q^* \circ \ell}^*(w) = q^{**}(w) + w\T b = \frac{1}{2\rho}\norm{w-\lambda}^2 + w\T b.
  }
  It follows that
  \eqe{ \label{23481}
    \inf_{w}\sup_{x} \Gamma_{\lambda}(w, x) &= \inf_{w}\pa{\frac{1}{2\rho}\norm{w-\lambda}^2 + w\T b + \sup_{x}\pa{-w\T A x - f(x)}} \\
    &= -\sup_{w}\pa{-\pa{q^* \circ \ell}^*(w) - f^*\pa{-A\T w}}.
  }
  Recall that $f$ and $q^* \circ \ell$ are both closed proper convex. By applying \cite[Corollary 31.2.1]{rockafellar1970convex} to \cref{23480,23481}, we can conclude that \cref{21450} holds if
  \eqe{ \label{23580}
    A\ri{\dom{f}} \cap \ri{\dom{\pa{q^* \circ \ell}}} \neq \emptyset.
  }
  Since $f$ is proper, $\dom{f} \neq \emptyset$, then $\ri{\dom{f}} \neq \emptyset$ \cite[Theorem 6.2]{rockafellar1970convex}. Additionally, since $\ri{\dom{\pa{q^* \circ \ell}}} = \ri{\R^p} = \R^p$, \cref{23580} holds, and consequently, \cref{21450} also holds.
  Note that
  \eqe{
    \arg\min_{w}\pa{\frac{1}{2\rho}\norm{w-\lambda}^2 - \dotprod{w, A x - b}} = \lambda + \rho(A x - b).
  }
  Combining \cref{00160,21450,23480}, we can finally obtain
  \eqe{
    M_{\rho (-\phi)}(\lambda) &= \sup_x - f(x) -\dotprod{\lambda, A x - b} - \frac{\rho}{2}\norm{A x - b}^2 \\
    &= -\phi_{\rho}(\lambda), \ \forall \lambda \in \R^p.
  }
\end{proof}

With \cref{1951,dom_phi}, we are now ready to prove \cref{augmented_dual}.
\begin{appendix_proof}[\cref{augmented_dual}]
  Given \cref{f}, $f^*$ is closed proper convex by \cite[Theorem E.1.1.2]{hiriart2004fundamentals}; consequently, $-\phi$ is closed and convex. Since $\phi(\lambda) = \inf_x\set{f(x) + \dotprod{\lambda, A x - b}}$ and $f$ is proper, it follows that $-\phi(\lambda) > -\infty$ for all $\lambda \in \R^p$. Combining this with the assumption that $\dom{\phi} \neq \emptyset$, we conclude that $-\phi$ is closed proper convex.

  According to \cref{moreau,1951}, we immediately obtain (1) and (2). A direct result of (1) is that $\inte{\dom \phi_{\rho}} = \inte{\R^p} = \R^p$. As $f$ is closed proper convex, so is $f_{\rho}$, which implies that $f^*_{\rho}$ is also closed proper convex \cite[Theorem E.1.1.2]{hiriart2004fundamentals}. Then, applying \cref{subdiff,0142} to \cref{1522} gives
  \eqe{ \label{0106}
    \partial \phi_{\rho}(\lambda) = A\partial f_{\rho}^*\pa{-A\T\lambda} - b, \ \forall \lambda \in \inte{\dom \phi_{\rho}} = \R^p.
  }
  From (2) we know that $\phi_{\rho}$ is differentiable everywhere, hence \cref{0106} implies two facts: (1) $\partial f_{\rho}^*\pa{-A\T\lambda} \neq \emptyset$; (2) $A\partial f_{\rho}^*\pa{-A\T\lambda}$ has only one element. Let $x^+$ be any element in $\partial f_{\rho}^*\pa{-A\T\lambda}$, we have $\nabla \phi_{\rho} = A x^+ - b$. By \cref{conjugate_subdiff}, we have
  \eqe{
    -A\T\lambda \in \partial f_{\rho}(x^+),
  }
  which implies that $x^+$ is a solution to $\min_{x \in \R^d} \cL_{\rho}(x, \lambda)$, as established by \cref{subdiff} and \cite[Theorem 27.1]{rockafellar1970convex}.
\end{appendix_proof}

\section{Conclusion}
In this paper, we establish the smoothness of the augmented Lagrangian dual function $\phi_{\rho}$ and guarantee the existence of solutions to the primal subproblem $\min_{x \in \R^d} \cL_{\rho}(x, \lambda)$ for any $\lambda \in \R^p$. Crucially, these fundamental properties hold for any general closed proper convex function $f$, provided that the standard dual domain is nonempty, i.e., $\dom{\phi} \neq \emptyset$. These results substantially weaken the stringent assumptions typically required in the literature, providing a more general and solid theoretical foundation for augmented Lagrangian-based methods.

\bibliographystyle{IEEEtran}
\bibliography{../public/bib}
\end{document}